\documentclass[11 pt]{amsart}
\usepackage{latexsym,amscd,amssymb, graphicx, amsthm, bm, young, amsbsy}  
\usepackage{arydshln}
\usepackage{tikz}

\usepackage[margin=1in]{geometry}

\numberwithin{equation}{section}

\newtheorem{theorem}{Theorem}[section]
\newtheorem{proposition}[theorem]{Proposition}

\newtheorem{example}[theorem]{Example}
\newtheorem{remark}[theorem]{Remark}

\newtheorem{defn}[theorem]{Definition}
\theoremstyle{definition}

\newcommand{\maj}{{\mathrm {maj}}}

\newcommand{\GL}{{\mathrm {GL}}}

\newcommand{\WComp}{{\mathrm {WComp}}}

\newcommand{\Comp}{{\mathrm {Comp}}}
\newcommand{\Hilb}{{\mathrm {Hilb}}}
\newcommand{\grFrob}{{\mathrm {grFrob}}}
\newcommand{\des}{{\mathrm {des}}}

\newcommand{\SYT}{{\mathrm {SYT}}}

\newcommand{\Frob}{{\mathrm {Frob}}}

\newcommand{\shape}{{\mathrm {shape}}}

\newcommand{\symm}{{\mathfrak{S}}}

\newcommand{\CC}{{\mathbb {C}}}
\newcommand{\QQ}{{\mathbb {Q}}}
\newcommand{\ZZ}{{\mathbb {Z}}}

\newcommand{\xx}{{\mathbf {x}}}
\newcommand{\II}{{\mathbf {I}}}
\newcommand{\TT}{{\mathbf {T}}}

\begin{document}

\title[Cyclic sieving and orbit harmonics]
{Cyclic sieving and orbit harmonics}

\author{Jaeseong Oh}
\address
{Department of Mathematical Sciences \newline \indent
Seoul National University \newline \indent
1 Gwanak-ro, Gwanak-gu,
Seoul, 08826, South Korea}
\email{jaeseong$\_$oh@snu.ac.kr}

\author{Brendon Rhoades}
\address
{Department of Mathematics \newline \indent
University of California, San Diego \newline \indent
9500 Gilman Dr., La Jolla, CA, 92093, USA}
\email{bprhoades@math.ucsd.edu}

\begin{abstract}
Orbit harmonics is a tool in combinatorial representation theory which promotes the
(ungraded) action
 of a linear group $G$ on a finite set $X$ to a graded action of $G$ on a polynomial ring quotient
 by viewing $X$ as a $G$-stable point locus in $\CC^n$.
 The cyclic sieving phenomenon is a notion in enumerative combinatorics which encapsulates 
 the fixed-point structure of the action of a finite cyclic group $C$ on a finite set $X$ in terms of root-of-unity
 evaluations of an auxiliary polynomial $X(q)$.
We apply orbit harmonics to prove cyclic sieving results.
\end{abstract}

\maketitle

\section{Introduction}
\label{Introduction}

Let $X$ be a finite set with an action of a finite cyclic group $C = \langle c \rangle$
and let $\omega = \exp(2 \pi i / |C|)$.
Let $X(q) \in \ZZ_{\geq 0}[q]$ be a polynomial with nonnegative integer coefficients. The triple
$(X, C, X(q))$ {\em exhibits the cyclic sieving phenomenon} \cite{RSW} if for all $r \geq 0$ we have
\begin{equation}
|X^{c^r}| = | \{ x \in X \,:\, c^r \cdot x = x \} | = X(\omega^r) = [ X(q) ]_{q = \omega^r}.
\end{equation}
More generally, if $X$ carries an action of a product $C_1 \times C_2 = \langle c_1 \rangle \times \langle c_2 \rangle$
of two finite cyclic groups and $X(q,t) \in \ZZ_{\geq 0}[q,t]$, the triple $(X, C_1 \times C_2, X(q,t))$ 
exhibits the {\em bicyclic sieving phenomenon} \cite{BRS} if for all $r, s \geq 0$ we have
\begin{equation}
|X^{(c_1^r, c_2^s)}| = | \{ x \in X \,:\, (c_1^r, c_2^s) \cdot x = x | = X(\omega_1^r, \omega_2^s)
\end{equation}
where $\omega_1 = \exp(2 \pi i / |C_1|)$ and $\omega_2 = \exp(2 \pi i / |C_2|)$.
In typical sieving results, $X$ is a set of combinatorial objects, the operators $c, c'$ act on $X$ by 
natural combinatorial actions, and $X(q)$ or $X(q,t)$ are generating functions for natural (bi)statistics 
on the set $X$.

Although ostensibly in the domain of enumerative combinatorics, the most desired proofs  of
CSPs are algebraic.
One seeks a representation-theoretic model for the action of $C$ on $X$ by finding a
$\CC$-vector space $V$ carrying an action of a group $G$ and possessing a distinguished basis
$\{ e_x \,:\, x \in X \}$ indexed by elements of the set $X$. The action of the generator $c \in C$
on $X$ is modeled by a group element $g \in G$ which satisfies 
$g \cdot e_x = e_{c \cdot x}$ for all $x \in X$. 
If $\chi: G \rightarrow \CC$ is the character of the $G$-module $V$, then
\begin{equation}
|X^{c^r}| = \mathrm{trace}_V(g^r) = \chi(g^r)
\end{equation}
 for all $r \geq 0$, transferring the enumerative problem of 
counting $|X^{c^r}|$ to the algebraic problem of calculating $\chi(g^r)$.  
These algebraic proofs are desired over brute force enumerative proofs because they 
give representation-theoretic insight about why a sieving result should hold.

In this article we use the orbit harmonics method of zero-dimensional
algebraic geometry to prove CSPs.
The results proven in this fashion will have the `nice' representation-theoretic proofs as 
outlined above.
This approach unifies various CSPs coming from
actions on word-like objects and  `quotients' thereof.
The  idea is to model the set $X$ geometrically as a finite point locus in $\CC^n$.
The relevant algebra has roots in (at least) the work of Kostant \cite{Kostant} and goes as follows.

The polynomial ring $\CC[\xx_n] := \CC[x_1, \dots, x_n]$ may be naturally viewed as the coordinate
ring of polynomial functions $f: \CC^n \rightarrow \CC$.
This identification gives rise to an action of the general linear group $\GL_n(\CC)$ on $\CC[\xx_n]$ by 
linear substitutions:
\begin{equation*}
g \cdot f(v) := f(g^{-1} \cdot v) \text{ for all $g \in \GL_n(\CC), f \in \CC[\xx_n],$ and $v \in \CC^n$.}
\end{equation*}
By restriction, any subgroup of $\GL_n(\CC)$ also acts on $\CC[\xx_n]$.

Let $X \subseteq \CC^n$ be a finite set of points which is closed under the action of $W \times C$ where
\begin{itemize}
\item $W \subseteq \GL_n(\CC)$ is a (finite) complex reflection group and
\item $C$ is a finite cyclic group acting on $\CC^n$ by root-of-unity scaling.
\end{itemize}
Let 
\begin{equation}
\II(X) := \{ f \in \CC[\xx_n] \,:\, f(v) = 0 \, \text{ for all $v \in X$} \}
\end{equation}
 be the ideal of polynomials in $\CC[\xx_n]$ which vanish on $X$.
Since  $X$ is finite, Lagrange Interpolation affords a $\CC$-algebra isomorphism
\begin{equation}
\label{first-orbit-iso}
\CC[X] \cong \CC[\xx_n]/\II(X)
\end{equation}
where $\CC[X]$ is the algebra of all functions $X \rightarrow \CC$.
Since $X$ is $W \times C$-stable, \eqref{first-orbit-iso} is also 
an isomorphism of ungraded $W \times C$-modules.
 
 For any nonzero polynomial $f \in \CC[\xx_n]$, let $\tau(f)$ be the highest degree component of $f$.
 That is, if $f = f_d + \cdots + f_1 + f_0$ with $f_i$ homogeneous of degree $i$ and $f_d \neq 0$, we set
 $\tau(f) := f_d$.  Given our locus $X$ with ideal $\II(X)$ as above, we define a homogeneous 
 ideal $\TT(X)$ by
 \begin{equation}
 \TT(X) := \langle \tau(f) \,:\, f \in \II(X), \, \, f \neq 0 \rangle \subseteq \CC[\xx_n].
 \end{equation}
 The ideal $\TT(X)$ is the {\em associated graded} ideal of $\II(X)$ and is often denoted
 $\mathrm{gr \,} \II(X)$.  By construction $\TT(X)$ is homogeneous and stable under $W \times C$.
 The isomorphism \eqref{first-orbit-iso} extends to an isomorphism of $W \times C$-modules
 \begin{equation}
\label{second-orbit-iso}
\CC[X] \cong \CC[\xx_n]/\II(X) \cong \CC[\xx_n]/\TT(X)
\end{equation}
where $\CC[\xx_n]/\TT(X)$ has the additional structure of a graded $W \times C$-module on which
$C$ acts by scaling in each fixed degree.

The action of $W \times C$ on $X$ coincides
\footnote{up to duality, but permutation representations are self-dual} 
with the $W \times C$-action on the natural basis $\{ e_x \,:\, x \in X \}$ of $\CC[X]$ 
of indicator functions $e_x: X \rightarrow \CC$ given by
\begin{equation}
e_x(y) = \begin{cases}
1 & x = y, \\
0 & \text{otherwise.}
\end{cases}
\end{equation}
If the graded 
$W$-isomorphism type of $\CC[\xx_n]/\TT(X)$ is known, 
the isomorphism \eqref{second-orbit-iso}
and Springer's theorem on regular elements \cite{Springer} (see Theorem~\ref{springer-theorem}) give
bicyclic sieving results for the set $X$
under product groups of the form $C' \times C$ where $C' \subseteq W$ is the subgroup generated by a 
{\em regular element} in $W$.
Furthermore, if $G \subseteq W$ is any subgroup, the set $X/G$ of $G$-orbits in $X$ carries a residual
action of $C$, and the Hilbert series of the $G$-invariant subspace
$(\CC[\xx_n]/\TT(X))^G$ is a cyclic sieving polynomial for the action of $C$ on $X/G$.
By varying
\begin{itemize}
\item the choice of point locus $X$ and
\item the choice of subgroup $G$ of $W$
\end{itemize}
a variety of CSPs can be obtained.
This method has been used \cite{ARR, D} to prove CSPs before; the purpose of this paper 
is to make its approach and utility explicit.

The procedure
$X \leadsto \CC[\xx_n]/\TT(X)$
which promotes the (ungraded) locus $X$ to the graded module $\CC[\xx_n]/\TT(X)$
is known as {\em orbit harmonics}.
Generators for the ideal $\TT(X)$ may be found by computer from the point set $X$ as follows.
The ideal $\II(X)$ may be expressed as either an intersection
or a product
\begin{equation}
\II(X) = \bigcap_{(\alpha_1, \dots, \alpha_n) \in X} \langle x_1 - \alpha_1, \dots, x_n - \alpha_n \rangle
= \prod_{(\alpha_1, \dots, \alpha_n) \in X} \langle x_1 - \alpha_1, \dots, x_n - \alpha_n \rangle
\end{equation}
of ideals corresponding to the points $(\alpha_1, \dots, \alpha_n)$ belonging
to $X \subseteq \CC^n$ . If $\{ g_1, g_2, \dots, g_r \}$ is a Gr\"obner basis
for $\II(X)$ with respect to any graded monomial order $\prec$ (i.e. we have $m \prec m'$ whenever $m, m'$ 
are monomials in $\CC[\xx_n]$ with $\deg m < \deg m'$), then 
$\TT(X)$ will be generated by $\{ \tau(g_1), \tau(g_2), \dots, \tau(g_r) \}$.
While this facilitates the investigation by computer of a graded quotient 
$\CC[\xx_n]/\TT(X)$ corresponding to a point locus $X$, a generating set of 
$\TT(X)$ obtained in this way can be messy and not so enlightening. Finding a nice generating set of 
$\TT(X)$ is often necessary to understand the graded $W$-isomorphism type of 
$\CC[X]/\TT(X)$.

In this paper we will focus mostly on the reflection group $W = \symm_n$ acting by 
coordinate permutation on $\CC^n$.
This setting has received substantial attention in algebraic combinatorics.
By making an appropriate choice of an $\symm_n$-stable locus $X$,
orbit harmonics has produced graded $\symm_n$-modules 
which give algebraic models for various intricate objects in
symmetric function theory \cite{GP, Griffin, HRS}.
It is our hope that this article inspires future connections between orbit harmonics 
and cyclic sieving.

We use  orbit harmonics  to reprove and unify a variety of known cyclic sieving results
\cite{ARR, BRS, BER, RSW, RhoadesFixed}, prove some cyclic sieving results which seem to have 
escaped notice, and give new proofs of some results \cite{Springer, MoritaNakajima} which are not 
{\em per se} in the field of cyclic sieving.
It would be interesting to see if our methods apply to the notion of {\em dihedral sieving}
due to Swanson \cite{Swanson} (see also \cite{RS}).

The remainder of this paper is organized as follows.
In {\bf Section~\ref{Background}} we give background on complex reflection groups and
the representation theory of $\symm_n$.
In {\bf Section~\ref{SMN}} we describe how orbit harmonics gives a new perspective 
on classical results of Springer and Morita-Nakajima.
We also state our main tool for proving sieving results from point loci
(Theorem~\ref{sieving-generator}).
In {\bf Section~\ref{Functional}} we apply Theorem~\ref{sieving-generator} to 
point loci corresponding to arbitrary, injective, and surjective functions between finite sets.
In {\bf Section~\ref{Other}} we apply Theorem~\ref{sieving-generator} to other combinatorial loci and
conclude in {\bf Section~\ref{Conclusion}} with possible future directions.

\section{Background}
\label{Background}

\subsection{Combinatorics}
A {\em weak composition of $n$} is a finite sequence $\mu = (\mu_1, \mu_2, \dots )$
of nonnegative integers which sum to $n$. We write $\mu \models_0 n$ to indicate that 
$\mu$ is a weak composition of $n$ and denote by $m_i(\mu)$ the multiplicity of $i$
as a part of $\mu$.
A {\em composition} of $n$ is a weak composition which contains only positive parts; we write
$\mu \models n$ to mean that $\mu$ is a composition of $n$.

A {\em partition} of $n$ is a composition $\lambda$ of $n$ consisting of weakly decreasing positive parts.
We write $\lambda \vdash n$ to indicate that $\lambda$ is a partition of $n$ and $|\lambda| = n$
for the sum of the parts of $\lambda$.
We also write $\ell(\lambda)$ for the number of parts of $\lambda$ and define the number
$b(\lambda) := \sum_{i \geq 1} (i-1) \cdot \lambda_i$.

Given a partition $\lambda = (\lambda_1, \lambda_2, \dots )$, the {\em Young diagram}
of $\lambda$ consists of $\lambda_i$ left-justified cells in row $i$. 
The Ferrers diagram of $(3,2,2) \vdash 7$ is shown on the left below.
A {\em tableau} $T$ of shape $\lambda$ is a filling $T: \lambda \rightarrow \ZZ_{> 0}$ of these boxes
with positive integers.
The {\em content} of a tableau $T$ is the weak composition $(\mu_1, \mu_2, \dots )$ where $\mu_i$ is the 
multiplicity of $i$ in $T$ and the {\em shape} $\shape(T)$ of $T$ is the partition $\lambda$.
A tableau $T$ is {\em semistandard} if its entries increase weakly across rows and strictly down
columns.  A semistandard tableau of shape $(3,2,2)$ and content $(2,0,2,2,1)$ is shown in the middle below.
The {\em Kostka number} $K_{\lambda,\mu}$ is the number of semistandard tableaux of shape $\lambda$
with content $\mu$.
A {\em standard} tableau is a semistandard tableau with content $(1, 1, \dots)$.
A standard tableau of shape $(3,2,2)$ is shown on the right below.

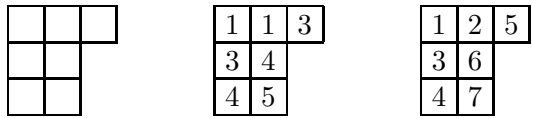
\begin{center}
\begin{young}
 & & \cr
 & \cr
 & \cr
\end{young}  \quad \quad \quad
\begin{young}
1 & 1 & 3 \cr
3 & 4 \cr
4 & 5
\end{young} \quad \quad \quad
\begin{young}
1 & 2 & 5 \cr
3 & 6 \cr
4 & 7
\end{young}
\end{center}

A {\em descent} in a word $w = w_1 \dots w_n$ over the positive integers is an index $i$ with 
$w_i > w_{i+1}$.  The {\em descent number} $\des(w)$ is the number of descents in $w$ and the
 {\em major index} $\maj(w)$ is the sum of the descents in $w$.
 Analogously,
 if $T$ is a standard tableau with $n$ boxes, an index $1 \leq i \leq n-1$ is a {\em descent} of $T$ if 
$i$ appears in a strictly higher row than $i+1$ in $T$. 
The {\em descent number} $\des(T)$ is
the number of descents in $T$
and the {\em major index} $\maj(T)$ be the sum of the descents in $T$.
The standard tableau $T$ shown above has descents $2, 3, 5,$ and $6$ so that $\des(T) = 4$ and
$\maj(T) = 2 + 3 + 5 + 6 = 16$.

The {\em fake degree polynomial}
$f^{\lambda}(q)$
corresponding to a partition $\lambda \vdash n$ is the generating function for major index over the 
set $\SYT(\lambda)$ of standard tableaux of shape $\lambda$:
\begin{equation}
f^{\lambda}(q) := \sum_{T \in \SYT(\lambda)} q^{\maj(T)}.
\end{equation}
The polynomial $f^{\lambda}(q)$ may be efficiently computed using the $q$-hook formula
\begin{equation}
f^{\lambda}(q) = q^{b(\lambda)} \frac{[n]!_q}{\prod_{c \in \lambda} [h_c]_q}
\end{equation}
where the product is over the cells $c$ in the Young diagram of $\lambda$ and $h_c$ is the hook
length at the cell $c$.  Here and throughout we use the standard $q$-analogs of numbers,
factorials, binomial, and multinomial coefficients:
\begin{equation}
\begin{cases}
[n]_q := 1 + q + \cdots + q^{n-1} \\ 
[n]!_q := [n]_q [n-1]_q \cdots [1]_q 
\end{cases} \quad
\begin{cases}
{n \brack k}_q := \frac{[n]!_q}{[k]!_q \cdot [n-k]!_q} \\
{n \brack \mu_1, \dots, \mu_r}_q  := \frac{[n]!_q}{[\mu_1]!_q \cdots [\mu_r]!_q}
\end{cases}
\end{equation}
where $\mu = (\mu_1, \dots, \mu_r)$ is a weak composition of $n$.

\subsection{Symmetric Functions}
We denote by $\Lambda = \bigoplus_{d \geq 0} \Lambda_d$ the graded ring of symmetric functions 
in an infinite variable set $\xx = (x_1, x_2, \dots )$ over the ground field $\CC(q)$.
Here $\Lambda_d$ denotes the 
subspace of $\Lambda$ consisting of homogeneous symmetric functions of degree $d$.
Two important elements of $\Lambda_d$ are the {\em homogeneous} and {\em elementary}
symmetric functions
\begin{equation}
h_d(\xx) := \sum_{i_1 \leq \cdots \leq i_d} x_{i_1} \cdots x_{i_d} \quad \quad \text{and} \quad \quad
e_d(\xx) := \sum_{i_1 < \cdots < i_d} x_{i_1} \cdots x_{i_d}.
\end{equation}
By restricting $h_d(\xx)$ and $e_d(\xx)$ to a finite variable set $\xx_n = \{x_1, \dots, x_n \}$,
we obtain the homogeneous and elementary
symmetric polynomials
$h_d(\xx_n)$ and $e_d(\xx_n)$.

Bases of $\Lambda_n$ are indexed by partitions of $n$. For a partition $\lambda \vdash n$, we let
\begin{equation}
h_{\lambda}(\xx), \quad e_{\lambda}(\xx), \quad s_{\lambda}(\xx), \quad \text{and} \quad
 \widetilde{H}_{\lambda}(\xx;q)
\end{equation}
denote the associated {\em homogeneous symmetric function}, {\em elementary symmetric function, Schur
function,} and {\em Hall-Littlewood symmetric function}.  For any partition
$\lambda = (\lambda_1 \geq \lambda_2 \geq \cdots )$ the $h$- and $e$-functions are defined by 
\begin{equation}
h_{\lambda}(\xx) := h_{\lambda_1}(\xx) h_{\lambda_2}(\xx) \cdots \quad \quad \text{and} \quad \quad
e_{\lambda}(\xx) := e_{\lambda_1}(\xx) e_{\lambda_2}(\xx) \cdots
\end{equation}
and the Schur function is given by
\begin{equation}
s_{\lambda}(\xx) := \sum_T \xx^T
\end{equation}
where the sum is over all semistandard tableaux $T$ of shape $\lambda$ and $\xx^T$ is shorthand for
the monomial
$x_1^{\alpha_1} x_2^{\alpha_2} \cdots $ where $\alpha_i$ is the number of $i$'s in the tableau $T$.
The $\widetilde{H}_{\lambda}(\xx;q)$ may be defined as 
\begin{equation}
\widetilde{H}_{\lambda}(\xx;q) =  
\sum_{\nu} \widetilde{K}_{\nu,\lambda}(q) \cdot s_{\nu}(\xx)
\end{equation}
where $\widetilde{K}_{\nu,\lambda}(q)$ is the {\em (modified) Kostka-Foulkes polynomial}
which is the generating function of the {\em cocharge} statistic on the semistandard tableaux of shape $\nu$
and content $\lambda$. We will not use the (somewhat involved) cocharge statistic explicitly in this paper; see
e.g. \cite{RhoadesFixed} for its definition.

\subsection{Representation theory of $\symm_n$}
A {\em class function} on a finite group $G$ is a map $\varphi: G \rightarrow \CC$ which is constant
on conjugacy classes. 
The set $R(G)$ of all class functions $G \rightarrow \CC$ forms a $\CC$-algebra under pointwise addition
and multiplication. We let $\langle -, - \rangle_G$ be the standard inner product on these class functions:
\begin{equation}
\langle \varphi, \psi \rangle_G := \frac{1}{|G|} \sum_{g \in G} \varphi(g) \cdot \overline{\psi(g)}.
\end{equation}
If $V, W$ are finite-dimensional $G$-modules with characters $\chi_V, \chi_W: G \rightarrow \CC$,
we extend this notation by setting $\langle V, W \rangle_G := \langle \chi_V, \chi_W \rangle_G$.

Irreducible representations of the symmetric group $\symm_n$ are in one-to-one correspondence
with partitions $\lambda \vdash n$. If $\lambda$ is a partition, we let $S^{\lambda}$ be the corresponding
irreducible module. If $V$ is any finite-dimensional $\symm_n$-module, there are unique 
multiplicities $c_{\lambda}$ so that $V \cong \bigoplus_{\lambda \vdash n} c_{\lambda} S^{\lambda}$.
The {\em Frobenius image} of $V$ is the symmetric function
\begin{equation}
\Frob(V) := \sum_{\lambda \vdash n} c_{\lambda} \cdot s_{\lambda}(\xx). 
\end{equation}
For example, if $\mu \vdash n$ and 
$\symm_{\mu} := \symm_{\mu_1} \times \symm_{\mu_2} \times \cdots $ is the corresponding parabolic
subgroup of $\symm_n$, the {\em coset representation}
\begin{equation}
M^{\mu} := \CC[\symm_n/\symm_{\mu}] = {\bf 1} \uparrow_{\symm_{\mu}}^{\symm_n}
\end{equation}
has Frobenius
image $h_{\mu}(\xx)$.

We will consider graded vector spaces and representations.
If $V = \bigoplus_{d \geq 0} V_d$ is any graded vector space, its {\em Hilbert series} is
 \begin{equation}
 \Hilb(V;q) = \sum_{d \geq 0} \dim (V_d) \cdot q^d.
 \end{equation}
 If
$V = \bigoplus_{d \geq 0} V_d$ is a graded $\symm_n$-module, its {\em graded Frobenius image} is
\begin{equation}
\grFrob(V; q) = \sum_{d \geq 0} \Frob(V_d) \cdot q^d.
\end{equation}

\subsection{Complex reflection groups} The general linear group $\GL_n(\CC)$ acts naturally on $V := \CC^n$.
An element $t \in \GL_n(\CC)$ is a {\em reflection} if its fixed space
$V^t := \{v \in V \,:\, t \cdot v = v \}$ has codimension one in $V$.
A finite subgroup $W \subseteq \GL_n(\CC)$ is a {\em reflection group} if it is generated by reflections.


Let $W \subseteq \GL_n(\CC)$ be a complex reflection group.  The full general linear group 
$\GL_n(\CC)$ acts on the polynomial ring $\CC[\xx_n] := \CC[x_1, \dots, x_n]$ by linear substitutions,
and by restriction $\CC[\xx_n]$ is a graded $W$-module.  Let $\CC[\xx_n]^W_+ \subseteq \CC[\xx_n]$ 
be the subspace of $W$-invariants with vanishing constant term and let 
$\langle \CC[\xx_n]^W_+ \rangle \subseteq \CC[\xx_n]$ be the ideal generated by this subspace.
The {\em coinvariant ring} attached to $W$ is the quotient 
\begin{equation}
R_W := \CC[\xx_n]/\langle \CC[\xx_n]^W_+ \rangle. 
\end{equation} 
This is a graded $W$-module.

For any irreducible $W$-module $U$, the {\em fake degree polynomial} $f^U(q)$ is the graded multiplicity
of $U$ in the coinvariant ring.  That is, we have
\begin{equation}
f^U(q) := \sum_{d \geq 0} m_{U,d} \cdot q^d
\end{equation}
where $m_{U,d}$ is the multiplicity of $U$ in the $W$-module given by the degree $d$ piece $(R_W)_d$ 
of $R_W$. When $W = \symm_n$ is the symmetric group of $n \times n$ permutation matrices, the irreducible
representations of $W$ correspond to partitions $\lambda \vdash n$, and $f^{S^{\lambda}}(q) = f^{\lambda}(q)$
specializes to our earlier definition.

Let $W \subseteq \GL_n(\CC)$ be a complex reflection group acting on $V = \CC^n$. An element
$c \in W$ is a {\em regular element} if it possesses an eigenvector $v \in V$ which has full $W$-orbit,
or equivalently the stabilizer
 $W_v := \{ w \in W \,:\, w \cdot v = v \}$ consists of the identity $e \in W$ alone.
 Such an eigenvector $v$ is called a {\em regular eigenvector} and if $\omega \in \CC^{\times}$
 is such that $w \cdot v = \omega v$, the order of $w \in W$ will equal the order of $\omega$
 in the multiplicative group $\CC^{\times}$.
 For example, when $W = \symm_n$ a permutation $w \in W$ is a regular element if and only if it is a power
 of an $n$-cycle or an $(n-1)$-cycle.

\subsection{Regular sequences}
A length $n$ sequence of polynomials $f_1, \dots, f_n \in \CC[\xx_n]$
is a {\em regular sequence} if for all $1 \leq i \leq n$, the map 
$(-) \times f_i: \CC[\xx_n]/\langle f_1, \dots f_{i-1} \rangle \rightarrow \CC[\xx_n]/\langle f_1, \dots, f_{i-1} \rangle$
of multiplication by $f_i$ is injective. If $f_1, \dots, f_n$ is a regular sequence, we have a short exact sequence
\begin{equation}
0 \rightarrow \CC[\xx_n]/\langle f_1, \dots, f_{i-1} \rangle \xrightarrow{(-) \times f_i}
\CC[\xx_n]/\langle f_1, \dots, f_{i-1} \rangle \twoheadrightarrow
\CC[\xx_n]/\langle f_1, \dots, f_{i-1}, f_i \rangle \rightarrow 0
\end{equation}
for all $i$.  If the $f_i$ are homogeneous, this implies that 
\begin{equation}
\Hilb(\CC[\xx_n]/\langle f_1, \dots, f_n \rangle; q) = [\deg f_1]_q \cdots [\deg f_n]_q
\end{equation}
and in particular $\dim(\CC[\xx_n]/\langle f_1, \dots, f_n \rangle) = \deg f_1 \cdots \deg f_n$.

A useful criterion for deciding whether a sequence of polynomials is regular is as follows.
Let $f_1, \dots, f_n \in \CC[\xx_n]$ be length $n$ sequence of homogeneous polynomials of 
positive degree.
This sequence is 
regular if and only if the locus in $\CC^n$ cut out by 
$f_1 = \cdots = f_n = 0$ consists of the origin $\{0\}$ alone.

\section{Theorems of Springer and Morita-Nakajima}
\label{SMN}

\subsection{Springer's theorem on regular elements}
\label{springer-subsection}
Before applying orbit harmonics to prove sieving results, 
we state representation-theoretic results
of Springer \cite{Springer} and Morita-Nakajima \cite{MoritaNakajima}.
which will be useful in our combinatorial work.
We explain how orbit harmonics may be used to prove these results.

Let $W \subseteq \GL_n(\CC)$ be a complex reflection group and let $c \in W$ be a regular element
with regular eigenvector $v \in \CC^n$ whose eigenvalue is $\omega \in \CC^{\times}$.
Let $C = \langle c \rangle$ be the cyclic subgroup of $W$ generated by $c$.  We regard the coinvariant ring
$R_W$ is a graded $W \times C$-module, where $W$ acts by linear substitutions and the generator
$c \in C$ sends each variable $x_i$ to $\omega x_i$.

\begin{theorem} (Springer \cite{Springer})
\label{springer-theorem}
Consider the action of $W \times C$ on $W$ given by $(u, c^r) \cdot w := uwc^{-r}$.
The corresponding
permutation representation $\CC[W]$ is isomorphic to $R_W$ as an ungraded $W \times C$-module.
\end{theorem}

\begin{remark}
\label{springer-remark}
Theorem~\ref{springer-theorem}
gives a way to compute the irreducible characters of $W$ on 
the subgroup $C$ generated by $c$. As explained 
in \cite[Prop. 4.5]{Springer}, if $U$ is an irreducible $W$-module with character
$\chi: W \rightarrow \CC$, then 
\begin{equation}
\chi(c^r) = \mathrm{trace}_U(c^r) = f^{U^*}(\omega^r) = [f^{U^*}(q)]_{\omega^r}
\end{equation}
where $f^{U^*}(q) \in \CC[q]$ is the fake degree polynomial attached to the dual  
$U^* = \mathrm{Hom}_{\CC}(U, \CC)$ of the representation $U$.
\end{remark}

We describe an argument of Kostant \cite{Kostant} which 
proves Theorem~\ref{springer-theorem} using orbit harmonics.
This argument will be used to give an orbit harmonics proof 
of a result \cite[Thm. 1.4]{BRS} of Barcelo, Reiner, and Stanton (see Theorem~\ref{springer-bicsp}).

We let $C$ act on $\CC^n$ by the rule $c \circ v' := \omega^{-1} v'$ for all $v' \in V$.
The corresponding action of $C$ on $\CC[\xx_n]$ by linear substitutions sends $x_i$ to $\omega x_i$ 
for all $i$, just like the $C$-action on $R_W$ in Theorem~\ref{springer-theorem}.
Furthermore, this action of $C$ on $\CC^n$ commutes with the natural action $(w, v') \mapsto w \cdot v'$ 
of $W$, so we may
regard $\CC^n$ as a $W \times C$-module in this way.

Define the {\em Springer locus} to be the $W$-orbit of the regular eigenvector $v$ of $c$:
\begin{equation}
W \cdot v := \{ w \cdot v \,:\, w \in W \} \subseteq \CC^n
\end{equation}
The locus
$W \cdot v$ is certainly closed under the action of $W$; we claim that $X$ is also
closed under the $\circ$-action of $C$.
Indeed, for any $w \in W$ we have 
\begin{equation*}
c \circ (w \cdot v) = \omega^{-1} (w \cdot v) = w \cdot (\omega^{-1} v) = w \cdot (c^{-1} \cdot v)  = 
(wc^{-1}) \cdot v \in W \cdot v.
\end{equation*}
Since the $\circ$ action of $C$ and the $\cdot$ action of $W$ on $\CC^n$ commute,
we may regard $X$ as a $W \times C$-set. 
An inspection of the above chain of equalities and the regularity of $v$ shows that
the map $w \mapsto w \cdot v$ furnishes a 
$W \times C$-equivariant bijection 
\begin{equation}
\label{ungraded-springer-isomorphism}
W \xrightarrow{ \, \sim \, } W \cdot v
\end{equation}
where the action of $W \times C$ on $W$ is as in Theorem~\ref{springer-theorem}.

Chevalley \cite{Chevalley} proved that there exist 
algebraically independent $W$-invariant polynomials
$f_1, \dots, f_n$ of homogeneous positive degree such that 
$\CC[\xx_n]^W = \CC[f_1, \dots, f_n]$. Furthermore, we have isomorphisms of ungraded $W$-modules
\begin{equation}
R_W = \CC[\xx_n]/\langle f_1, \dots, f_n \rangle \cong \CC[W].
\end{equation}

We claim that the invariant polynomials $f_1, \dots, f_n$ generate the ideal $\TT(X)$. Indeed,
for any $1 \leq i \leq n$, let $f_i(v) \in \CC$ be the value of $f_i$ on the regular eigenvector $v \in \CC^n$.
The $W$-invariance of $f_i$ implies that $f_i - f_i(v) \in \II(W \cdot v)$, and taking the top degree component
gives $f_i \in \TT(W \cdot v)$, so that $\langle f_1, \dots, f_n \rangle \subseteq \TT(W \cdot v)$.
On the other hand, 
\begin{equation}
\dim ( \CC[\xx_n] / \langle f_1, \dots, f_n \rangle ) = \dim \CC[W] = |W| = |W \cdot v| =
\dim ( \CC[\xx_n] / \TT(W \cdot v) ),
\end{equation}
so that
\begin{equation}
\label{springer-ideal-equality}
\langle f_1, \dots, f_n \rangle = \TT(W \cdot v).
\end{equation}

We use the ideal equality \eqref{springer-ideal-equality} to prove Theorem~\ref{springer-theorem}.
Since the defining ideal $\langle \CC[\xx_n]^W_+ \rangle$ of the coinvariant ring $R_W$
is generated by $f_1, \dots, f_n$, orbit harmonics furnishes isomorphisms of ungraded $W \times C$-modules
\begin{equation}
R_W = 
\CC[\xx_n]/\langle f_1, \dots, f_n \rangle =
\CC[\xx_n]/\TT(W \cdot v) \cong
\CC[W \cdot v] \cong \CC[W]
\end{equation}
where the last isomorphism used
the $W \times C$-equivariant bijection \eqref{ungraded-springer-isomorphism}.
This completes the orbit harmonics proof of Theorem~\ref{springer-theorem}.

\subsection{A theorem of Morita-Nakajima via Orbit Harmonics}
\label{MN-subsection}
In this subsection we consider the case of the symmetric $W = \symm_n$.
Throughout this subsection, we fix a 
weak composition $\mu = (\mu_1, \dots, \mu_k)$ of $n$ with $k$
parts which has cyclic symmetry
of order $a$ for some $a \mid k$.  
That is, we have 
$\mu_i = \mu_{i+a}$ for all $i$ with subscripts interpreted modulo $k$.
Let $c$ be an arbitrary but fixed generator of the cyclic group $\ZZ_{k/a}$.
Morita and Nakajima proved \cite{MoritaNakajima} a variant
of Springer's Theorem~\ref{springer-theorem} as follows. 

Let $W_{\mu}$ be the family of length $n$ words $w_1 \dots w_n$ over the alphabet $[k]$
in which the letter $i$ appears $\mu_i$ times. 
The set $W_{\mu}$ carries an action of 
$\symm_n \times \ZZ_{k/a}$ where $\symm_n$ acts by 
subscript permutation and  $\ZZ_{k/a}$ acts by
$c: w_1 \dots w_n \mapsto (w_1 + a) \cdots (w_n + a)$ where letter values are interpreted modulo $k$.
Extending by linearity, the space $\CC[W_{\mu}]$ is a $\symm_n \times \ZZ_{k/a}$-module.

Let $I_{\mu} \subseteq \CC[\xx_n]$ be the {\em Tanisaki ideal} attached to the 
composition $\mu$ and let $R_{\mu} := \CC[\xx_n]/I_{\mu}$ be the corresponding 
{\em Tanisaki quotient ring}.  The ring $R_{\mu}$ is a graded $\symm_n$-module which may be described
in three equivalent ways.
\begin{enumerate}
\item  Let $\mathcal{F \ell}_n$ be the variety of complete flags 
$V_{\bullet} = (0 = V_0 \subset V_1 \subset \cdots \subset V_n = \CC^n)$ of subspaces of $\CC^n$.
If $X_{\mu} \in \GL_n(\CC)$ is a unipotent operator of Jordan type $\mu$,
The {\em Springer fiber} $\mathcal{B}_{\mu}$ is the closed subvariety of 
$\mathcal{F \ell}_n$ consisting of flags $V_{\bullet}$ which satisfy $X_{\mu} V_i = V_i$ for all $i$.
The cohomology of $\mathcal{B}_{\mu}$ may be presented \cite{Tanisaki} as 
\begin{equation*}
H^{\bullet}(\mathcal{B}_{\mu}; \CC) = R_{\mu}.
\end{equation*}
Although the variety $\mathcal{B}_{\mu}$ is not stable under the action of $\symm_n \subseteq \GL_n(\CC)$,
there is a {\em Springer representation} of $\symm_n$ on the cohomology ring 
$H^{\bullet}(\mathcal{B}_{\mu}; \CC)$.
\item Tanisaki \cite{Tanisaki} and Garsia-Procesi \cite{GP} gave explicit generators for the defining ideal 
$I_{\mu}$ of $R_{\mu}$.  These generators are certain elementary symmetric polynomials 
$e_d(S)$ in subsets $S$ of the variable set $\{x_1, \dots, x_n\}$ which depend on $\mu$.
This presentation makes it clear that $R_{\mu}$ is closed under the action of $\symm_n$;
we will make no explicit use of it here. Garsia and Procesi used this presentation to show that the 
graded Frobenius image of $R_{\mu}$ is a Hall-Littlewood function:
\begin{equation}
\label{hl-char}
\grFrob(R_{\mu}; q) = \widetilde{H}_{\mu}(\xx;q).
\end{equation}
Here we interpret $\widetilde{H}_{\mu}(\xx;q) := \widetilde{H}_{\mathrm{sort}(\mu)}(\xx;q)$ where
$\mathrm{sort}(\mu)$ is the partition obtained by sorting $\mu$ into weakly decreasing order.
\item  Let $\alpha_1, \dots, \alpha_k \in \CC$ be distinct complex numbers. We may consider the set 
$W_{\mu} \subseteq \CC^n$ as a point locus by identifying
\begin{equation*}
w_1 \dots w_n \leftrightarrow (\alpha_{w_1}, \dots, \alpha_{w_n}).
\end{equation*}
We have $I_{\mu} = \TT(W_{\mu})$ as ideals in $\CC[\xx_n]$. Since $W_{\mu}$ is closed under 
the coordinate permuting action of $\symm_n$, this makes it clear that 
$R_{\mu} = \CC[\xx_n]/I_{\mu} = \CC[\xx_n]/\TT(W_{\mu})$ is $\symm_n$-stable.
\end{enumerate}
The orbit harmonics interpretation (3) of $R_{\mu}$ was used by Garsia and Procesi \cite{GP} to derive 
Equation~\eqref{hl-char}.

Let $\omega := \exp(2 a \pi i / k)$ be a primitive  $(k/a)^{th}$ root-of-unity. We extend the graded
$\symm_n$-action on $R_{\mu}$ to a graded $\symm_n \times \ZZ_{k/a}$-action by letting 
the distinguished generator
$c \in \ZZ_{k/a}$ scale by $\omega^d$ in homogeneous degree $d$.

\begin{theorem}
(Morita-Nakajima \cite[Theorem 13]{MoritaNakajima})
\label{mn-theorem}
We have an isomorphism of ungraded $\symm_n \times \ZZ_{k/a}$-modules
\begin{equation*}
\CC[W_{\mu}] \cong R_{\mu}.
\end{equation*}
\end{theorem}

When $\mu = (1^n)$, Theorem~\ref{mn-theorem} reduces to Theorem~\ref{springer-theorem}
when $W = \symm_n$ at the regular element $(1, 2, \dots, n) \in \symm_n$.
The proof of Theorem~\ref{mn-theorem} given in \cite{MoritaNakajima} involves
tricky symmetric function manipulations involving the Hall-Littlewood polynomials
$\widetilde{H}_{\mu}(\xx;q)$ when $q$ is specialized to a root of unity.
The work of
\cite{MoritaNakajima} 
relies on further intricate symmetric function identities due to Lascoux-Leclerc-Thibon \cite{LLT}.
Orbit harmonics gives an easier and more conceptual proof.

\begin{proof}
Let $\zeta := \exp(2 \pi i / k)$ be a primitive $k^{th}$ root-of-unity with $\zeta^a = \omega$.
In the interpretation (3) of $R_{\mu}$ described above, take the parameters $\alpha_1, \dots, \alpha_k$
defining the point locus $W_{\mu} \subseteq \CC^n$ to be $\alpha_j := \zeta^j$.  
If we let our distinguished generator $c$ of $\ZZ_{k/a}$ act on $\CC^n$ as scaling by $\omega$,
the set $W_{\mu}$ is closed under the action of the linear group $\symm_n \times \ZZ_{k/a}$.
As discussed above, Garsia and Procesi \cite{GP} used orbit harmonics to give an isomorphism
\begin{equation}
\CC[W_{\mu}] \cong \CC[\xx_n]/\TT(W_{\mu}) = R_{\mu}
\end{equation}
of ungraded $\symm_n \times \ZZ_{k/a}$-modules.
\end{proof}

\subsection{Loci and Sieving}
Our main `generating theorem' for sieving results is as follows.
For any group $W$, write $\mathrm{Irr}(W)$ for the family of (isomorphism classes of)
irreducible $W$-modules.

\begin{theorem}
\label{sieving-generator}
Let $W \subseteq \GL_n(\CC)$ be a complex reflection group, let $c' \in W$ be a regular element,
let $C' = \langle c' \rangle$ be the subgroup of $W$ generated by $c'$, and let 
$\omega := \exp(2 \pi i / k) \in \CC^{\times}$.
Let $C = \langle c \rangle \cong \ZZ_k$ be a cyclic group of order $k$ and consider the action
of $W \times C$ on $\CC^n$ where $c$ scales by $\omega$ and $W$ acts by left 
multiplication.

Let $X \subseteq \CC^n$ be a finite point set which is closed under the action of $W \times C$.

\begin{enumerate}
\item  Suppose that for $d \geq 0$, the isomorphism type of the degree $d$ piece of 
$\CC[\xx_n]/\TT(X)$ is given by
\begin{equation*}
( \CC[\xx_n]/\TT(X) )_d \cong \bigoplus_{U \in \mathrm{Irr}(W)} U^{\oplus m_{U,d}}.
\end{equation*}
The triple $(X, C \times C', X(q,t))$ exhibits the bicyclic sieving phenomenon where 
\begin{equation*}
X(q,t) = \sum_{U \in \mathrm{Irr}(W)} m_U(q) \cdot f^{U^*}(t).
\end{equation*}
where $m_U(q) := \sum_{d \geq 0} m_{U,d} \cdot q^d$.
\item Let $G \subseteq W$ be a subgroup. The set $X/G$ of $G$-orbits in $X$ carries a natural
$C$-action and the triple $(X/G, C, X(q))$ exhibits the cyclic sieving phenomenon where
\begin{equation*}
X(q) = \Hilb( (\CC[\xx_n]/\TT(X))^G; q).
\end{equation*}
\end{enumerate}
\end{theorem}

\begin{proof}
Applying orbit harmonics to the action of $W \times C$ on $X$ 
yields an isomorphism of ungraded $W \times C$-modules
\begin{equation}
\label{ungraded-isomorphism}
\CC[X] \cong \CC[\xx_n]/\TT(X).
\end{equation}
Let $\zeta := \exp(2 \pi i / n)$.
To prove (1), apply Theorem~\ref{springer-theorem} (and Remark~\ref{springer-remark}) to get that 
for any $r, s \geq 0$, the trace of 
 $(c^r, (c')^s) \in C \times C'$ acting on $\CC[\xx_n]/\TT(X)$ is 
$\sum_{U \in \mathrm{Irr}(W)} m_U(\omega^r) \cdot f^{U^*}(\zeta^s) = X(\omega^r, \zeta^s)$.
By the isomorphism \eqref{ungraded-isomorphism}, this coincides with the number of fixed
points of $(c^r, (c')^s)$ acting on $X$, completing the proof of (1).

For (2), we take $G$-invariants of both sides of the isomorphism \eqref{ungraded-isomorphism}
to get an isomorphism of $C$-modules
\begin{equation}
\label{g-invariants}
\CC[X/G] \cong (\CC[\xx_n]/\TT(X))^G.
\end{equation}
Since $C$ acts on the graded vector space $(\CC[\xx_n]/\TT(X))^G$ by root-of-unity scaling, 
we see that the trace of $c^r$ on \eqref{g-invariants} is 
$[\Hilb( (\CC[\xx_n]/\TT(X))^G; q)]_{q = \omega^r} = X(\omega^r)$ (for the 
right-hand side) or the number 
of orbits in $X/G$ fixed by $c^s$ (for the left-hand side), finishing the proof.
\end{proof}

\begin{remark}
We will mostly apply Theorem~\ref{sieving-generator} in the case $W = \symm_n$.
In this context $\mathrm{Irr}(W)$ coincides with the family of partitions $\lambda \vdash n$
and  the $f^{U^*}(t)$ appearing in Theorem~\ref{sieving-generator} (1) may be replaced by 
$f^{\lambda}(t)$.
\end{remark}

In order to use Theorem~\ref{sieving-generator} to prove a sieving result involving a set 
$X$, we must 
\begin{itemize}
\item realize the relevant action on $X$ in terms of an action on a point locus
in $\CC^n$, or a quotient thereof, and
\item calculate the graded isomorphism type of
$\CC[\xx_n]/\TT(X)$,
or the Hilbert series of $(\CC[\xx_n]/\TT(X))^G$.
\end{itemize}
As we shall see, this program varies in difficulty depending on the combinatorial action in question.

\section{The Functional Loci}
\label{Functional}

Rota's {\em Twelvefold Way}  is a foundational
result in enumeration and involves counting functions between finite  sets which 
are arbitrary, injective, or surjective.
Inspired by this, the loci considered in this section correspond 
to arbitrary, injective, and surjective functions between finite sets.

\begin{defn}
\label{functional-locus-definition}
Given integers $n$ and $k$, set $\omega := \exp(2 \pi i / k)$.
We define the following three point sets in $\CC^n$:
\begin{align}
X_{n,k} &:= \{ (a_1, \dots, a_n) \,:\, a_i \in \{ \omega, \omega^2, \dots, \omega^k \} \} \\
Y_{n,k} &:= \{ (a_1, \dots, a_n) \in X_{n,k} \,:\, \text{$a_1, \dots, a_n$ are distinct} \} \\
Z_{n,k} &:= \{ (a_1, \dots, a_n) \in X_{n,k} \,:\, \{a_1, \dots, a_n\} = \{\omega, \omega^2, \dots \omega^k \} \}
\end{align}
\end{defn}

Observe that $Y_{n,k} = \varnothing$ if $k < n$ and $Z_{n,k} = \varnothing$ if $n < k$. When $n = k$,
we have the identification of $Y_{n,k} = Z_{n,k}$ with permutations in $\symm_n$.  
Each of these sets is closed under the action of $\symm_n \times \ZZ_k$, where $\ZZ_k$ scales by 
$\omega$.

\subsection{The Quotient Rings}
In this subsection we give describe generating sets for the homogeneous ideals $\TT(X)$
corresponding to the quotients in Definition~\ref{functional-locus-definition} and describe the
graded $\symm_n$-isomorphism types of the quotients
$\CC[\xx_n]/\TT(X)$. This is easiest for the case of $X_{n,k}$ corresponding to arbitrary 
functions $[n] \rightarrow [k]$.

\begin{proposition}
\label{x-quotient-structure}
Let $n, k \geq 0$. The quotient ring $\CC[\xx_n]/\TT(X_{n,k})$ has presentation
\begin{equation}
\CC[\xx_n]/\TT(X_{n,k}) = \CC[\xx_n] / \langle x_1^k, \dots, x_n^k \rangle
\end{equation}
and we have
\begin{equation}
\grFrob(\CC[\xx_n]/\TT(X_{n,k}); q) = 
\sum_{\substack{\lambda_1 < k \\ \ell(\lambda) \leq n}} q^{|\lambda|} \cdot h_{m(\lambda)}
\end{equation}
where the sum is over all partitions $\lambda$ with largest part size $< k$ and length $\leq n$.
Here $m(\lambda)$ is the partition obtained by rearranging the nonzero elements of 
$(n-\ell(\lambda), m_1(\lambda), m_2(\lambda), \dots, m_{k-1}(\lambda))$ into weakly decreasing order.
\end{proposition}

\begin{proof}
Let $1 \leq i \leq n$. For any $(a_1, \dots, a_n) \in X_{n,k}$ we have
$a_i \in \{ \omega, \omega^2, \dots, \omega^k\}$ which means that
$(x_i - \omega)(x_i - \omega^2) \cdots (x_i - \omega^k) \in \II(X_{n,k})$. Taking the highest
degree component, we have $x_i^k \in \TT(X_{n,k})$.
We conclude that $\langle x_1^k, \dots, x_n^k \rangle \subseteq \TT(X_{n,k})$. 
On the other hand, we know that 
\begin{equation*}
\dim \CC[\xx_n]/\TT(X_{n,k}) = |X_{n,k}| = k^n
\end{equation*}
and since $\dim \CC[\xx_n] / \langle x_1^k, \dots, x_n^k \rangle = k^n$, this proves the first assertion.

For the second assertion, observe that
$\{ x_1^{b_1} \cdots x_n^{b_n} \,:\, 0 \leq b_i < k \}$ is a basis for the quotient
$\CC[\xx_n] / \langle x_1^k, \dots, x_n^k \rangle$. The action of $\symm_n$ on these 
monomials decomposes this set into orbits indexed by the partitions appearing in the sum. 
For each partition $\lambda$, the corresponding orbit lies in degree $|\lambda|$ and has 
Frobenius image $h_{m(\lambda)}$.
\end{proof}

We turn our attention to the locus $Y_{n,k}$ corresponding to injective functions
$[n] \rightarrow [k]$. 
The idea is to use regular sequences to reduce to the case where $k = n$. 

\begin{proposition}
\label{y-quotient-structure}
For any $n \leq k$ we have
\begin{equation}
\CC[\xx_n]/\TT(Y_{n,k}) = \CC[\xx_n] / \langle h_{k-n+1}(\xx_n), \dots, h_{k-1}(\xx_n), h_k(\xx_n) \rangle
\end{equation}
and
\begin{equation}
\grFrob(\CC[\xx_n]/\TT(Y_{n,k}); q) = {k \brack n}_q \cdot \sum_{T \in \SYT(n)} q^{\maj(T)} \cdot s_{\shape(T)}
\end{equation}
where the sum is over all standard tableaux $T$ with $n$ boxes.
\end{proposition}

\begin{proof}
We first show that $h_d(\xx_n) \in \TT(Y_{n,k})$ for any $d > k-n$.  Indeed, introduce a new variable 
$t$ and consider the quotient
\begin{equation}
\frac{(1 - \omega t)(1 - \omega^2 t) \cdots (1 - \omega^k t)}{(1 - x_1 t) (1 - x_2 t) \cdots (1 - x_n t)}
= \sum_{d \geq 0} \sum_{a + b = d} (-1)^a \cdot e_a(\omega, \omega^2, \dots, \omega^k) \cdot
h_b(\xx_n) \cdot t^d.
\end{equation}
Whenever $(x_1, \dots, x_n) \in Y_{n,k}$, the $n$ factors in the denominator will cancel with $n$
factors in the numerator, yielding a polynomial in $t$ of degree $k-n$. If $d > k-n$, taking the 
coefficient of $t^d$ on both sides yields
$\sum_{a + b = d} (-1)^a \cdot e_a(\omega, \omega^2, \dots, \omega^k) \cdot
h_b(\xx_n) \in \II(Y_{n,k})$ so that $h_d(\xx_n) \in \TT(Y_{n,k})$.

It is known that the sequence $h_{k-n+1}(\xx_n), \dots, h_{k-1}(\xx_n), h_k(\xx_n)$ of polynomials
in $\CC[\xx_n]$ is regular. One way to see this is to check that the locus in $\CC^n$ cut out by these 
polynomials consists only of the origin $\{ 0 \}$.  Indeed, the identities 
\begin{equation}
h_d(x_1, \dots, x_i) = x_i \cdot h_{d-1}(x_1, \dots, x_i) + h_d(x_1, \dots, x_{i-1})
\end{equation}
mean that the system 
\begin{equation}
h_{k-n+1}(x_1, \dots, x_n) = 
h_{k-n+2}(x_1, \dots, x_n) = \cdots = h_{k-1}(x_1, \dots, x_n) = h_k(x_1, \dots, x_n) = 0
\end{equation}
is equivalent to the system
\begin{equation}
h_{k-n+1}(x_1, \dots, x_n) =
h_{k-n+2}(x_1, \dots, x_{n-1}) = \cdots = h_{k-1}(x_1, x_2) = h_k(x_1) = 0.
\end{equation}
The latter system is triangular and may be solved to yield the solution set $\{0\}$.

Since $h_{k-n+1}(\xx_n), \dots, h_{k-1}(\xx_n), h_k(\xx_n)$ is a regular sequence, we see that
\begin{multline}
\dim \CC[\xx_n]/\langle h_{k-n+1}(\xx_n), \dots, h_{k-1}(\xx_n), h_k(\xx_n) \rangle \\ = 
(k-n+1) \cdots (k-1) \cdot k = |Y_{n,k}| = \dim \CC[\xx_n]/\TT(Y_{n,k})
\end{multline}
which forces $\langle h_{k-n+1}(\xx_n), \dots, h_{k-1}(\xx_n), h_k(\xx_n) \rangle = \TT(Y_{n,k})$.
Furthermore, since the exact sequences 
\begin{multline}
0 \rightarrow \CC[\xx_n]/\langle h_{k-n+1}(\xx_n), \dots, h_{k-n+i-1}(\xx_n) \rangle
\xrightarrow{(-) \times h_{k-n+i}(\xx_n)} 
\CC[\xx_n]/\langle h_{k-n+1}(\xx_n), \dots, h_{k-n+i-1}(\xx_n) \rangle \\ \twoheadrightarrow 
\CC[\xx_n]/\langle h_{k-n+1}(\xx_n), \dots, h_{k-n+i-1}(\xx_n), h_{k-n+i}(\xx_n) \rangle \rightarrow 0
\end{multline}
involve $\symm_n$-equivariant maps, we see that 
\begin{equation}
\grFrob(\CC[\xx_n]/\TT(Y_{n,k}); q) = \grFrob(\CC[\xx_n];q) \times (1 - q^{k-n+1}) \cdots (1 - q^{k-1})(1-q^k)
\end{equation}
for all $k \geq n$
and in particular
\begin{multline}
\grFrob(\CC[\xx_n]/\TT(Y_{n,k});q) = {k \brack n}_q \cdot \grFrob(\CC[\xx_n]/\TT(Y_{n,n}); q)  \\ =
{k \brack n}_q \cdot \grFrob(\CC[\xx_n]/\langle h_1(\xx_n), \dots, h_n(\xx_n) \rangle; q) =
{k \brack n}_q \cdot \sum_{T \in \SYT(n)} q^{\maj(T)} \cdot s_{\shape(T)}
\end{multline}
where the final equality is due to
 Lusztig (unpublished)
and Stanley \cite{Stanley}.
\end{proof}

The surjective locus $Z_{n,k}$ was studied by Haglund, Rhoades, and Shimozono \cite{HRS}
and is the most difficult functional locus to analyze.
We quote their results here.

\begin{proposition}
\label{z-quotient-structure} (Haglund-Rhoades-Shimozono \cite{HRS})
For any $k \leq n$ we have
\begin{equation}
\CC[\xx_n]/\TT(Z_{n,k}) = \CC[\xx_n]/\langle x_1^k, x_2^k, \dots, x_n^k, 
e_n(\xx_n), e_{n-1}(\xx_n), \dots, e_{n-k+1}(\xx_n) \rangle
\end{equation}
and
\begin{equation}
\grFrob(\CC[x_1, \dots, x_n]/\TT(Z_{n,k}); q) = 
\sum_{T \in \SYT(n)} q^{\maj(T)} \cdot {n - \des(T) - 1 \brack n-k}_q \cdot
s_{\shape(T)}.
\end{equation}
\end{proposition}

All known proofs of Proposition~\ref{z-quotient-structure} involve intricate arguments 
such as Gr\"obner theory, a variant of Lehmer codes attached to points in $Z_{n,k}$,
and somewhat involved symmetric function theory.

\subsection{Bicyclic Sieving}
We describe the bicyclic sieving phenomena obtained by applying 
Theorem~\ref{sieving-generator} (1) to the loci $X_{n,k}, Y_{n,k},$ and $Z_{n,k}$.
Rather than actions on point sets, we phrase the result in terms of actions on words.

\begin{theorem}
\label{word-bicyclic-sieving}
Let $n$ and $k$ be positive integers. The following triples exhibit the bicyclic sieving phenomenon.

\begin{enumerate}
\item  The triple $(X_{n,k}, \ZZ_n \times \ZZ_k, X_{n,k}(q,t))$ where 
$X_{n,k}$ is the set of all length $n$ words $w_1 w_2 \dots w_n$ over the alphabet $[k]$,
the cyclic group $\ZZ_n$ acts by rotating positions
$w_1 w_2 \dots w_n \mapsto w_2 \dots w_n w_1$, the cyclic group $\ZZ_k$ 
acts by rotating values $w_1 w_2 \dots w_n \mapsto (w_1 + 1)(w_2 + 1) \cdots (w_n + 1)$
(interpreted modulo $k$), and
\begin{equation}
X_{n,k}(q,t) = \sum_{\substack{\mu \text{ a partition} \\ \ell(\mu) \leq n, \, \, \mu_1 \leq k}}
q^{|\mu|} \cdot {n \brack n - \ell(\mu), m_1(\mu), m_2(\mu), \dots }_t.
\end{equation}
\item The triple $(Y_{n,k}, \ZZ_n \times \ZZ_k, Y_{n,k}(q,t))$ where 
$Y_{n,k} \subseteq X_{n,k}$ is the subset of words with distinct letters, the group 
$\ZZ_n \times \ZZ_k$ acts by restricting its action on $X_{n,k}$, and
\begin{equation}
Y_{n,k}(q,t) = {k \brack n}_q \cdot \sum_{\lambda \vdash n} f^{\lambda}(q) \cdot f^{\lambda}(t).
\end{equation}
\item  The triple $(Z_{n,k}, \ZZ_n \times \ZZ_k, Z_{n,k}(q,t))$ where 
$Z_{n,k} \subseteq X_{n,k}$ is the subset of words in which every letter in $[k]$ appears,
the group $\ZZ_n \times \ZZ_k$ acts by restricting its action on $X_{n,k}$, and 
\begin{equation}
Z_{n,k}(q,t) = \sum_{T \in \SYT(n)} q^{\maj(T)} \cdot {n - \des(T) - 1 \brack n-k}_q \cdot
f^{\shape(T)}(t).
\end{equation}
\end{enumerate}
\end{theorem}


\begin{proof}
We prove these items in reverse order: (3), then (2), then (1).
Item (3) follows immediately by combining 
Theorem~\ref{sieving-generator} (1) and Proposition~\ref{z-quotient-structure}.

If we apply Theorem~\ref{sieving-generator} (1) to the locus $Y_{n,k}$ and use 
Proposition~\ref{y-quotient-structure}, we obtain a bicyclic sieving result with polynomial
\begin{equation}
{k \brack n} \cdot 
\sum_{T \in \SYT(n)} q^{\maj(T)} f^{\shape(T)}(t) = {k \brack n} \cdot 
\sum_{\lambda \vdash n} f^{\lambda}(t) \cdot f^{\lambda}(t) =
Y_{n,k}(q,t);
\end{equation}
this proves item (2) of this theorem.

Finally, if we apply Theorem~\ref{sieving-generator} (1) to the locus $X_{n,k}$ and
use Proposition~\ref{x-quotient-structure}, we obtain a bicyclic sieving result with polynomial
\begin{equation}
\label{z-first}
\sum_{\substack{\mu \text{ a partition} \\ \ell(\mu) \leq n, \, \, \mu_1 \leq k}}   \sum_{\lambda \vdash n}
q^{|\mu|}  \cdot  K_{\lambda, \mu} \cdot f^{\lambda}(t)
\end{equation}
 where we applied {\em Young's Rule}: for $\mu \vdash n$ we have
 $h_{\mu} = \sum_{\lambda \vdash n} K_{\lambda,\mu} s_{\lambda}$.
For fixed $\mu$,  we claim that
\begin{equation}
\label{z-second}
\sum_{\lambda \vdash n} K_{\lambda,\mu} \cdot f^{\lambda}(t) = 
{n \brack n - \ell(\mu), m_1(\mu), m_2(\mu), \dots }_t.
\end{equation}
Indeed, the left-hand side of Equation~\ref{z-second} is the generating function for the 
major index statistic over the set of words $w = w_1 \dots w_n$ with $m_i(\mu)$ copies of $i$.
The {\em Schensted correspondence} bijects such words $w$ with ordered pairs $(P, Q)$
of tableaux of the same shape with $n$ boxes such that 
\begin{itemize}
\item  $P$ is semistandard and $Q$ is standard, and
\item  if $w \mapsto (P, Q)$ then $\maj(w) = \maj(Q)$.
\end{itemize}
Since $f^{\lambda}(t)$ is the generating function for major index on standard tableaux of shape
$\lambda$, Equation~\eqref{z-second} follows.
Applying Equation~\eqref{z-second}, we see that the expression \eqref{z-first} equals the 
formula for $X_{n,k}(q,t)$ in the statement of the theorem.
\end{proof}

\subsection{The subgroup $G = \symm_n$}
For our first application of Theorem~\ref{sieving-generator} (2), we consider the case where 
$W = \symm_n$ and the subgroup
$G = \symm_n$ is the full symmetric group. 
We phrase our sieving results in terms of compositions. 

Let $\WComp_{n,k}$ be the family of weak compositions of $n$ of length $k$ and $\Comp_{n,k}$ 
be the set of compositions of $n$ of length $k$.
We have natural identifications of orbit sets 
\begin{equation}
X_{n,k}/\symm_n = \WComp_{n,k} \quad \quad
Y_{n,k}/\symm_n = {[k] \choose n} \quad \quad
Z_{n,k}/\symm_n = \Comp_{n,k}
\end{equation}
where ${[k] \choose n}$ is the family of $n$-element subsets of $[k]$.
The relevant sieving result reads as follows.

\begin{theorem}
\label{compositions-csp}
Let $n$ and $k$ be positive integers. The following triples exhibit the cyclic sieving phenomenon.
\begin{enumerate}
\item The triple $(\WComp_{n,k}, \ZZ_k, {n+k-1 \brack n}_q)$ where $\ZZ_k$ acts by rotation
$(\alpha_1, \alpha_2, \dots, \alpha_k) \mapsto (\alpha_2, \dots, \alpha_k, \alpha_1)$.
\item The triple $\left(  {[k] \choose n}, \ZZ_k, {k \brack n}_q \right)$ where $\ZZ_k$ acts on 
${ [k] \choose n}$ by increasing entries by $1$ modulo $k$.
\item The triple $(\Comp_{n,k}, \ZZ_k, {n-1 \brack k-1}_q)$ where  $\ZZ_k$ acts by rotation.
\end{enumerate}
\end{theorem}

Theorem~\ref{compositions-csp} (2) is the `first example' of the CSP of 
Reiner, Stanton, and White \cite{RSW}. 
Theorem~\ref{compositions-csp} (1) and (3) seem to have not been explicitly stated in the literature 
so far.

\begin{proof}
For any graded $\symm_n$-module $V$, the Hilbert series 
$\Hilb( V^{\symm_n}; q)$ of the $\symm_n$-invariant subspace of $V$ is the coefficient
of $s_{(n)}$ in $\grFrob(V;q)$.  
Since the coefficient of $s_{(n)}$ in $h_{\mu}$ is 1 for any partition $\mu \vdash n$, 
Proposition~\ref{x-quotient-structure} yields
\begin{equation}
\Hilb(\CC[\xx_n]/\TT(X_{n,k})^{\symm_n}; q) =
\sum_{\substack{\lambda \text{ a partition} \\ \lambda_1 < k, \, \, \ell(\lambda) \leq n}} q^{|\lambda|} =
{n+k-1 \brack n}_q
\end{equation}
and Theorem~\ref{sieving-generator} (2) finishes the proof of item (1).
For item (2), we extract the coefficient of $s_{(n)}$ in Proposition~\ref{y-quotient-structure} to get
\begin{equation}
\Hilb(\CC[\xx_n]/\TT(Y_{n,k})^{\symm_n}; q) = {k \brack n}_q
\end{equation}
and apply Theorem~\ref{sieving-generator} (2).  For item (3), extracting the coefficient of 
$s_{(n)}$ in Theorem~\ref{z-quotient-structure} yields
\begin{equation}
\Hilb(\CC[\xx_n]/\TT(Z_{n,k})^{\symm_n};q) = {n-1 \brack k-1}_q
\end{equation}
and Theorem~\ref{sieving-generator} (2) finishes the proof.
\end{proof}

\subsection{The subgroup $G = C_n$}

In this section we consider the subgroup $C_n$ of $\symm_n$ generated by the long 
cycle $(1, 2, \dots, n)$. Recall that a {\em necklace} of length $n$ over the bead set $[k]$
is an equivalence class of length $n$ sequences $(b_1, b_2, \dots, b_n)$ of beads in $[k]$
where two such sequences are considered equivalent if they differ by a cyclic shift.
We have the following three families of necklaces over the bead set $[k]$.
\begin{equation}
\begin{cases}
N^X_{n,k} := \{ \text{all length $n$ necklaces with bead set $[k]$} \} \\
N^Y_{n,k} := \{ \text{all length $n$ necklaces with bead set $[k]$ and no repeated beads} \} \\
N^Z_{n,k} := \{ \text{all length $n$ necklaces with bead set $[k]$ in which every bead appears} \}
\end{cases}
\end{equation}
Each of the sets $N^X_{n,k}, N^Y_{n,k}, N^Z_{n,k}$ carries an action of $\ZZ_k$ by bead color rotation:
\begin{equation}
(b_1, b_2, \dots, b_n) \mapsto (b_1 + 1, b_2 + 1, \dots, b_n + 1) \quad \quad 
\text{(letters interpreted modulo $k$)}
\end{equation}
and we have the orbit set identifications
\begin{equation}
X_{n,k}/C_n = N^X_{n,k} \quad \quad 
Y_{n,k}/C_n = N^Y_{n,k} \quad \quad
Z_{n,k}/C_n = N^Z_{n,k}.
\end{equation}

\begin{theorem}
\label{cn-reduction-theorem}
Let $n$ and $k$ be positive integers. The following triples exhibit the cyclic sieving phenomenon.
\begin{enumerate}
\item The triple $(N^X_{n,k}, \ZZ_k, N^X_{n,k}(q))$ where
\begin{equation*}
N^X_{n,k}(q) = \sum_{\substack{\text{$\mu$ a partition }\\ \ell(\mu) \leq n, \, \, \mu_1 < k}} q^{|\mu|} \cdot b_{\mu}
\end{equation*}
where $b_{\mu}$ is the number of length $n$ words $w = w_1 \dots w_n$  
 satisfying $n \mid \maj(w)$ containing $n - \ell(\mu)$ copies of $0$ and $m_i(\mu)$ copies of $i$ 
 for each $i > 0$.
\item The triple $(N^Y_{n,k}, \ZZ_k, N^Y_{n,k}(q))$ where
$$
N^Y_{n,k}(q) = {k \brack n}_q \cdot  \sum_{\lambda \vdash n} a_{\lambda, n} \cdot  f^{\lambda}(q)
$$
where $a_{\lambda, n}$ is the number of standard Young tableaux $T$ of shape $\lambda$ 
with $n \mid \maj(T)$.
\item The triple $(N^Z_{n,k}, \ZZ_k, N^Z_{n,k}(q))$ where
$$
N^Z_{n,k}(q) = \sum_{T \in \SYT(n)} q^{\maj(T)} \cdot {n - \des(T) - 1 \brack n - k}_q 
\cdot a_{\mathrm{shape}(T),n}.
$$
\end{enumerate}
\end{theorem}

\begin{proof}
For $0 \leq j \leq n-1$, let $V_j$ be the linear representation of $C_n$ on which 
$(1, 2, \dots, n-1)$ acts by $\exp(2 \pi i j / n)$.
Kra\'skiewicz and Weyman proved \cite{KW} that the induction of $V_j$ from $C_n$ to $\symm_n$
may be written
\begin{equation}
V_j \uparrow^{\symm_n} \cong \bigoplus_{\maj(T) \equiv j \text{ (mod $n$)}} S^{\mathrm{shape}(T)}
\end{equation}
where the direct sum is over all standard tableaux $T$ with $n$ boxes such that $\maj(T) \equiv j$
modulo $n$.  Specializing at $j = 0$, we see that the dimension of the $C_n$-fixed space of
the $\symm_n$-irreducible $S^{\lambda}$ is 
\begin{equation}
\langle {\bf 1}, S^{\lambda} \downarrow_{C_n} \rangle_{C_n} = 
\langle {\bf 1} \uparrow^{\symm_n}, S^{\lambda} \rangle_{\symm_n} = a_{\lambda, n}
\end{equation}
which proves (2) and (3).  For (1), apply the RSK bijection to see that 
\begin{equation}
\langle {\bf 1}, M^{m(\mu)} \downarrow_{C_n} \rangle_{C_n} = 
\sum_{\lambda \vdash n} K_{\lambda, m(\mu)} \cdot \langle {\bf 1}, S^{\lambda} \downarrow_{C_n} \rangle_{C_n}
= \sum_{\lambda \vdash n} K_{\lambda, m(\mu)} \cdot a_{\lambda,n} = b_{m(\mu),n}.
\end{equation}
\end{proof}

\subsection{The subgroup $G = H_r$}

In this subsection we assume $n = 2r$ is even and let $H_r \subseteq \symm_n$ be the 
subgroup generated by the permutations
\begin{equation*}
(1,2), (3,4), \dots, (n-1, n), (1,3)(2,4), (3,5)(4,6), \dots, (n-3,n-1)(n-2,n).
\end{equation*}
The group $H_r$ is isomorphic to 
the hyperoctohedral group of signed permutations of an $r$-element set.

We consider undirected graphs on the vertex set $[k]$ in which multiple edges and loops are allowed.
An {\em isolated vertex} in such a graph is an element $i \in [k]$ which is not incident to any edge
(so that a vertex which has a loop is not isolated).  We have the following families of graphs on
the vertex set $[k]$.
\begin{equation}
\begin{cases}
Gr^X_{n,k} := \{ \text{all $r$-edge graphs on $[k]$} \} \\
Gr^Y_{n,k} := \{ \text{all $r$-edge loopless graphs on $[k]$ in which each vertex has degree $\leq 1$} \} \\
Gr^Z_{n,k} := \{ \text{all $r$-edge graphs on $[k]$ with no isolated vertices} \}
\end{cases}
\end{equation}
Observe that when $k = n$, the set $Gr^Y_{n,n}$ may be identified with the family of 
perfect matchings on $[n]$.  

Each of the three sets $Gr^X_{n,k}, Gr^Y_{n,k},$ and $Gr^Z_{n,k}$ carries an action of $\ZZ_k$ by 
the vertex-rotating cycle $(1, 2, \dots, k)$. This $\ZZ_k$-action is compatible with the orbit-set 
identifications
\begin{equation}
X_{n,k}/H_r = Gr^X_{n,k} \quad \quad
Y_{n,k}/H_r = Gr^Y_{n,k} \quad \quad
Z_{n,k}/H_r = Gr^Z_{n,k}.
\end{equation}
In the following theorem, a partition $\lambda$ is called {\em even} if each of its parts
$\lambda_1, \lambda_2, \dots $ is even.

\begin{theorem}
\label{hr-reduction-theorem}
We have the following cyclic sieving triples.
\begin{enumerate}
\item  The triple $(Gr_{n,k}^X, \ZZ_k, Gr_{n,k}^X(q))$ exhibits the cyclic sieving phenomenon,
where 
$$Gr_{n,k}^X(q) = \sum_{\substack{\text{partitions $\mu$} \\ \mu \subseteq (k^n) }}
\sum_{\lambda \vdash n \text{ even}}
K_{\lambda, m(\mu)}  \cdot q^{|\mu|} \cdot f^{\lambda}(q).$$
Here $m(\mu) \vdash n$ is the partition obtained by writing the part multiplicities in 
$(\mu_1, \mu_2, \dots, \mu_n)$ in weakly decreasing order and $K_{\lambda, m(\mu)}$ is the 
Kotska number.
\item The triple $(Gr_{n,k}^Y, \ZZ_k, Gr_{n,k}^Y(q))$ exhibits the cyclic sieving phenomenon,
where
$$Gr_{n,k}^Y(q) = {k \brack n}_q \cdot \sum_{\substack{\lambda \vdash n \\ \lambda \text{ even}}} f^{\lambda}(q).$$
\item The triple $(Gr_{n,k}^Z, \ZZ_k, Gr_{n,k}^Z(q))$ exhibits the cyclic sieving phenomenon,
where
$$Gr_{n,k}^Z(q) = \sum_{\substack{T \in \SYT(n) \\ \mathrm{shape}(T) \text{ even}}} q^{\maj(T)} \cdot
{n - \des(T) - 1 \brack n - k}_q.$$
\end{enumerate}
\end{theorem}

This result should be compared with 
a result of Berget, Eu, and Reiner. In \cite[Theorem 5 (i)]{BER}
they prove a result equivalent to Theorem~\ref{hr-reduction-theorem} (1).
Their proof, like ours, uses the symmetric function operation of {\em plethysm}.

\begin{proof}
For any value of $d$, the wreath product $\symm_d \wr \symm_r$ embeds naturally inside the larger
 symmetric group $\symm_{dr}$.  The Frobenius image of the corresponding coset representation
 may be expressed using plethysm as 
\begin{equation}
\Frob({\bf 1} \uparrow_{\symm_d \wr \symm_r}^{\symm_{dr}}) =  \Frob(\CC[ \symm_{dr} / \symm_d \wr \symm_r ])
 = h_d[h_r].
\end{equation}
Finding the Schur expansion of $h_d[h_r]$ is difficult in general, and is closely related to
{\em Thrall's Problem}. In the special case $d = 2$, we may identify $H_r \cong \symm_2 \wr \symm_r$ and 
we have 
\begin{equation}
h_2[h_r] = \sum_{\substack{\lambda \vdash n \\ \lambda \text{ even}}} s_{\lambda}
\end{equation}
where $n = 2r$.
Applying Frobenius Reciprocity, for any $\lambda \vdash n$ the dimension of the $H_r$-fixed subspace of 
the $\symm_n$-irreducible $S^{\lambda}$ is given by the character inner product:
\begin{equation}
\dim (S^{\lambda})^{H_r} = \langle {\bf 1}, \chi^{\lambda} \downarrow^{\symm_n}_{H_r} \rangle_{H_r} = 
\langle {\bf 1} \uparrow_{H_r}^{\symm_n}, \chi^{\lambda} \rangle_{\symm_n} = \begin{cases}
1 & \text{$\lambda$ is even} \\
0 & \text{$\lambda$ is not even}
\end{cases}
\end{equation}
where we use ${\bf 1}$ for the trivial character of $H_r$.
Correspondingly, if $V$ is any graded $\symm_n$-module with graded Frobenius image 
\begin{equation}
\grFrob(V; q) = \sum_{\lambda \vdash n} c_{\lambda}(q) s_{\lambda},
\end{equation}
the Hilbert series of the $H_r$-fixed subspace will be
\begin{equation}
\Hilb(V^{H_r}; q) = \sum_{\substack{\lambda \vdash n \\ \lambda \text{ even}}} c_{\lambda}(q).
\end{equation}
The polynomials $Gr^X_{n,k}(q), Gr^Y_{n,k}(q), Gr^Z_{n,k}(q)$ are obtained in this way
from Propositions~\ref{x-quotient-structure}, \ref{y-quotient-structure}, and
\ref{z-quotient-structure}.
\end{proof}

\section{The Tanisaki and Springer Loci}
\label{Other}

\subsection{The Tanisaki Locus}
\label{Tanisaki}

Throughout this subsection, we fix 
a weak composition $\mu = (\mu_1, \dots, \mu_k) \models_0 n$ of $n$ into $k$ parts which satisfies
$\mu_i = \mu_{i+a}$ for all $i$, where indices are interpreted modulo $k$. 
If $\omega := \exp(2 \pi i/k)$, define the {\em Tanisaki locus} $X_{\mu} \subseteq \CC^n$ by
\begin{equation}
X_{\mu} := \{ (\alpha_1, \dots, \alpha_n) \,:\, \alpha_j = \omega^i \text{ for precisely $\mu_i$ values of $j$} \}.
\end{equation}
As discussed in Subsection~\ref{MN-subsection}, Garsia-Procesi \cite{GP} proved that 
$\TT(X_{\mu}) = I_{\mu}$ (the Tanisaki ideal) and 
$\grFrob(\CC[\xx_n]/\TT(X_{\mu}); q) = \widetilde{H}_{\mu}(\xx;q)$ (the Hall-Littlewood symmetric function).
We have the following bicyclic sieving result.

\begin{theorem}
\label{word-bicyclic}
Let $W_{\mu}$ be the set of length $n$ words $w_1 \dots w_n$ which contain $\mu_i$ copies of the letter $i$
for each $i = 1, 2, \dots k$.
The set $W_{\mu}$ carries an action of 
$\ZZ_n \times \ZZ_{k/a}$, where $\ZZ_n$ acts by word rotation
$w_1 w_2 \dots w_n \mapsto w_2 \dots w_n w_1$ and $\ZZ_{k/a}$ acts by
$w_1 \dots w_n \mapsto (w_1 + \ell(\mu)/a) \dots (w_n + \ell(\mu)/a)$ where letter values are interpreted
modulo $k$. The triple $(W_{\mu}, \ZZ_n \times \ZZ_{k/a}, X(q,t))$ exhibits the 
bicyclic sieving phenomenon, where
\begin{equation}
X_{\mu}(q,t) = \sum_{\lambda \vdash n} 
\widetilde{K}_{\lambda,\mathrm{sort}(\mu)}(q) \cdot f^{\lambda}(t)
\end{equation}
and $\mathrm{sort}(\mu)$ is the partition obtained by sorting the parts of $\mu$ into weakly decreasing order.
\end{theorem}

\begin{proof}
Apply Theorem~\ref{sieving-generator} (1) to the point locus $X_{\mu}$ and use the Schur expansion
\begin{equation}
\widetilde{H}_{\nu}(\xx;q) = \sum_{\lambda \vdash n} \widetilde{K}_{\lambda, \nu}(q) \cdot s_{\lambda}(\xx)
\end{equation}
of the Hall-Littlewood polynomials.
\end{proof}

Theorem~\ref{word-bicyclic} was proven in unpublished work of Reiner and White.
A proof of Theorem~\ref{word-bicyclic} using Theorem~\ref{mn-theorem} may by found in
\cite{RhoadesFixed}.

Given a subgroup $G \subseteq \symm_n$,
what happens when we apply Theorem~\ref{sieving-generator} (2) to $Y_{\mu}$? By reasoning analogous to that
in Section~\ref{Functional} we obtain the following CSPs.

\begin{example}
 When $G = \symm_n$, the locus $X_{\mu}$ is a single $G$-orbit.
The $\symm_n$-invariant part of $R_{\mu}$ is simply the ground field in degree 0, so we get the trivial CSP
$( \{ * \}, \ZZ_{k/a}, 1)$ for the action of $\ZZ_k$ on a one-point set.
\end{example}

\begin{example}
When $G = C_n$, we may identify $X_{\mu}/G$ with the family of $n$-bead necklaces
$(b_1, \dots, b_n)$ with $\mu_i$
copies of the bead of color $i$. The cyclic group $\ZZ_{k/a}$ acts on these necklaces by $a$-fold color rotation
$b_i \mapsto b_i + a$ and
we get a CSP $(X_{\mu}/G, \ZZ_{k/a}, X(q))$ where 
\begin{equation}
X(q) = \sum_{\lambda \vdash n} \widetilde{K}_{\lambda,\mathrm{sort}(\mu)}(q)
\cdot a_{\lambda,n}
\end{equation}
and $a_{\lambda,n}$ is the number of standard tableaux $T$ of shape $\lambda$ with $n \mid \maj(T)$.
\end{example}

\begin{example}
When $n = 2r$ is even and $G = H_r$, we may identify $X_{\mu}/G$ with the family of graphs
(loops and multiple edges permitted)
on the vertex set $[k]$ where the vertex $i$ has degree $\mu_i$ (here a loop contributes two to the degree
of its vertex).  The orbit set $X_{\mu}/G$ is acted upon by $\ZZ_{k/a}$ by $a$-fold vertex rotation, and the triple
$(X_{\mu}/G, \ZZ_{k/a}, X(q))$ exhibits the CSP where
\begin{equation}
X(q) = \sum_{\substack{\lambda \vdash n \\ \lambda \text{ even}}} 
\widetilde{K}_{\lambda,\mathrm{sort}(\mu)}(q).
\end{equation}
\end{example}

\subsection{The Springer Locus}
\label{Springer}
In this subsection we return to the setting of an arbitrary complex reflection group $W \subseteq \GL_n(\CC)$
acting on $V := \CC^n$.  
We fix a regular element $c \in W$ with regular eigenvector $v \in V$ and corresponding regular 
eigenvalue $\omega \in \CC$, so that $c \cdot v = \omega v$. We also let $C := \langle c \rangle$ be the subgroup
of $W$ generated by $c$.

Recall that the Springer locus is the $W$-orbit
$W \cdot v = \{ w \cdot v \,:\, w \in W \} \subseteq V$.
Subsection~\ref{springer-subsection} shows that the Springer locus is closed under the action
of the group $W \times C$, where $W$ acts by its natural action on $V$ and $C$ acts by the rule
$c: v' \mapsto \omega v'$ for all $v' \in V$.
(Note that this is different from the $\circ$-action $c \circ v' := \omega^{-1} v'$ of $C$ considered
in Subsection~\ref{springer-subsection}.)

\begin{theorem}
\label{springer-bicsp}
Let $c, c' \in W$ be regular elements and let $C = \langle c \rangle, C' = \langle c' \rangle$ be the cyclic 
subgroups which they generate. The product of cyclic groups $C \times C'$ acts on $W$ by the rule
$(c, c') \cdot w := c' w c$. The triple $(W, C \times C', W(q,t))$ exhibits the bicyclic sieving phenomenon where
\begin{equation}
W(q,t) := \sum_{U} f^U(q) \cdot f^{U^*}(t)
\end{equation}
and the sum is over all (isomorphism classes of) irreducible $W$-modules $U$.
\end{theorem}

\begin{proof}
The discussion in Subsection~\ref{springer-subsection} shows that the homogeneous quotient
$\CC[V]/\TT(W \cdot v)$ attached to the Springer locus $W \cdot v \subseteq V$ is given by the coinvariant ring
\begin{equation}
R_W = \CC[V]/\TT(W \cdot v).
\end{equation}
The map $W \rightarrow W \cdot v$ given by
$w \mapsto w \cdot v$ is a $W \times C$-equivariant bijection.
Indeed, the generator $c$ of the group $C$ acts on $w \cdot v \in W \cdot v$ by
\begin{equation*}
c: w \cdot v \mapsto \omega ( w \cdot v) = w \cdot (\omega v) = w \cdot (c \cdot v) = (wc) \cdot v
\end{equation*}
which agrees with the action of $C$ on $W$.
By definition, 
the fake degree polynomial $f^{U}(q)$ is the graded multiplicity of $U$ in the $W$-module $R_W$.
Now apply Theorem~\ref{sieving-generator} (1).
\end{proof}

Theorem~\ref{springer-bicsp} is a result of Barcelo, Reiner, and Stanton
\cite[Thm. 1.4]{BRS}. 
In \cite{BRS} the polynomial $W(q,t)$ is referred to as a {\em bimahonian distribution}.
More generally, Barcelo, Reiner, and Stanton consider `Galois twisted' actions 
of $C \times C'$ on $W$ as follows. Let $d$ be the order of the regular element $c' \in W$ and let 
$s$ be an integer coprime to $d$.  The group $C \times C'$ acts on $W$ by the rule
\begin{equation}
\label{twisted-action}
(c, c') := (c')^s \cdot w \cdot c.
\end{equation} 
If we let $\zeta := \exp(2 \pi i / d)$, there is a unique Galois automorphism
$\sigma \in \mathrm{Gal}(\QQ[\zeta]/\QQ)$ satisfying $\sigma(\zeta) = \zeta^s$. 
Furthermore, if $U$ is any $W$-module, there is a $W$-module $\sigma(U)$ obtained by 
applying $\sigma$ entrywise to the matrices representing group elements in the action of $W$ on $U$.
The operation $U \mapsto \sigma(U)$ preserves the irreducibility of $W$-modules.

In \cite[Thm. 1.4]{BRS} Barcelo, Reiner, and Stanton prove that $(W, C \times C', W^{\sigma}(q,t))$ exhibits
the biCSP where $C \times C'$ acts by the Galois-twisted action of \eqref{twisted-action} and 
$W^{\sigma}(q,t)$ is the {\em $\sigma$-bimahonian distribution} given by 
\begin{equation}
W^{\sigma}(q,t) := \sum_U f^U(q) \cdot f^{\sigma^{-1}(U)}(t)
\end{equation}
where $U$ runs over all isomorphism classes of irreducible $W$ modules.
This more general biCSP can also be proven using orbit harmonics; one 
observes that $(c')^s \in C'$ is a regular element in $W$ with regular eigenvalue 
$\zeta^s$ and applies Theorem~\ref{springer-bicsp} with $c' \mapsto (c')^s$.

\section{Conclusion}
\label{Conclusion}

In this paper we described the orbit harmonics method of proving CSPs by modeling the set $X$
in question as a point locus in a $\CC$-vector space $V$.
We applied our results to the functional loci, the Springer Locus, and the Tanisaki Locus.

The orbit harmonics method has implicitly been used to prove CSPs before; we refer the reader to \cite{ARR}
for all undefined terms in what follows.
Let $W$ be an irreducible real reflection group acting on its (complexified) reflection representation $V$.
In \cite{ARR}, Armstrong, Reiner, and Rhoades define a {\em parking locus} $V^{\Theta} \subseteq V$
of degree $(h+1)^{\dim V}$
which carries an action of $W \times \ZZ_h$ where $h$ is the Coxeter number of $W$ and $\ZZ_h$ acts on $V$
by root-of-unity scaling.
The locus $V^{\Theta}$ depends on a choice of  homogeneous system of parameters $\Theta$ of degree
$h+1$ carrying $V^*$. 

In \cite{ARR, RhoadesEvidence}, the orbit harmonics method is (implicitly)
applied to the parking locus $V^{\Theta}$ to prove the following CSP of Bessis and Reiner.
Let $\mathrm{NC}_W$ be the set of $W$-noncrossing partitions with its cyclic action of $\ZZ_h$ by generalized
rotation. (In type A$_{n-1}$, this is simply the set of noncrossing set partitions of $[n]$ with the rotation
action of $\ZZ_n$.)  
Bessis and Reiner proved \cite{BessisReiner}
that the triple $(\mathrm{NC}_W, \ZZ_h, \mathrm{Cat}_q(W))$ exhibits the CSP
where $\mathrm{Cat}_q(W)$ is the {\em $q$-$W$-Catalan number} given by 
$\mathrm{Cat}_q(W) = \prod_{i = 1}^{\dim V} \frac{[d_i + h]_q}{[d_i]_q}$ where $d_1, d_2, \dots $
are the invariant degrees of $W$.
On the level of orbit harmonics, this corresponds to looking at the $W$-orbits $V^{\Theta}/W$
in the parking locus $V^{\Theta}$.
Whereas the proof of Bessis-Reiner was a brute force enumration, the proof in \cite{ARR, RhoadesEvidence}
is algebraic (since it uses orbit harmonics).

There are many interesting point loci $X \subseteq V$ to which the method of orbit harmonics has been applied 
to study the quotient $\CC[V]/\TT(X)$.  This paper describes a technique for obtaining sieving results in
any such situation.
One interesting recent locus which we did not consider in this paper is due to Griffin \cite{Griffin}.
This locus $X$ is in bijection with a certain family of ordered set partitions and is a common generalization
of the Tanisaki locus studied in \cite{GP, Tanisaki} and the surjective functional locus studied in 
\cite{HRS}. 
Furthermore, while most of our loci were stable under the action of $\symm_n$, it is natural to ask search for enumerative results concerning loci closed under other groups $W$.
We leave the sieving study of other interesting combinatorial loci to future work.

\section{Acknowledgements}

J. Oh was supported by the Basic Science Research Program through the National
Research Foundation of Korea (NRF) funded by the 
Ministry of Education (2020R1A6A3A13076804).
B. Rhoades was partially supported by NSF Grant DMS-1500838 and DMS-1953781.
The authors are grateful to Vic Reiner and Josh Swanson for helpful conversations.

\end{document}